\documentclass[twoside,a4paper,11pt]{amsart}
\usepackage{amsmath,amssymb,amsthm,amscd}
\usepackage[]{latexsym,amssymb,amsmath,amsfonts, amsthm, verbatim}
\usepackage[all,cmtip,ps]{xy}
\usepackage[latin1]{inputenc}
\usepackage[english]{babel}

\usepackage{graphicx}

\newcommand{\add}{\mathrm{add}}

\newcommand{\M}{\mbox{\rm Mod-$R$}}

\newcommand{\Z}{\mathbb{Z}}

\DeclareMathOperator{\Hom}{Hom}
\DeclareMathOperator{\End}{End}
\DeclareMathOperator{\Ext}{Ext}
\DeclareMathOperator{\Tor}{Tor}

\newcommand{\Dcal}{\ensuremath{\mathcal{D}}}

\newcommand{\Ycal}{\ensuremath{\mathcal{Y}}}
\newcommand{\Xcal}{\ensuremath{\mathcal{X}}}

\newcommand{\tube}{\ensuremath{\mathbf{t}}}

\newcommand{\Tria}{\mathrm{Tria\,}}
\theoremstyle{plain}
\newtheorem{thm}{Theorem}[section]
\newtheorem{prop}[thm]{Proposition}
\newtheorem{lem}[thm]{Lemma}

\newtheorem{cor}[thm]{Corollary}
\newtheorem{ex}[thm]{Example}
\theoremstyle{definition}
\newtheorem{defn}[thm]{Definition}
\theoremstyle{remark}

\newcommand{\ra}{\rightarrow}

\newcommand{\ten}{\otimes}
\newcommand{\lten}{\overset{\mathbf{L}}{\ten}}

\newcommand{\mcx}{\mathcal{X}}

\begin{document}
\begin{center}
{\bf On the uniqueness of stratifications of derived module categories}
\bigskip

{\sc Lidia Angeleri  H\" ugel, Steffen Koenig, Qunhua
Liu}

\bigskip

(Version of \today)
\end{center}

\address{Lidia Angeleri  H\" ugel
\\ Dipartimento di Informatica - Settore Matematica
\\ Universit\`a degli Studi di Verona
\\ Strada Le Grazie 15 - Ca' Vignal 2
\\ I - 37134 Verona, Italy}
\email{lidia.angeleri@univr.it}
\address{Steffen Koenig, Qunhua Liu\\ Institut f\"ur Algebra und
Zahlentheorie \\ Universit\"at Stuttgart \\ Pfaffenwaldring 57 \\
D-70569 Stuttgart, Germany}
\email{skoenig@mathematik.uni-stuttgart.de,
qliu@mathematik.uni-stuttgart.de}
\thanks{First named author acknowledges
support by  MIUR, PRIN-2008 "Anelli, algebre, moduli e categorie" and by
the DGI and the
European Regional Development Fund, jointly, through Project
 MTM2005--00934, and by the Comissionat per Universitats i Recerca
of the Generalitat de Ca\-ta\-lunya, Project 2005SGR00206. }
\date{\today}

\bigskip
\begin{quote}
{\footnotesize {\bf Abstract.} Recollements of triangulated
categories may be seen as exact sequences of such categories.
Iterated recollements of triangulated categories are
analogues of geometric or topological stratifications and of composition
series of algebraic objects. We discuss
the question of uniqueness of such a stratification, up to
ordering and derived equivalence, for derived module categories.
The main result is a positive answer in the form of a Jordan H\"older
theorem for derived module categories of hereditary artin
algebras. We also provide examples of derived simple rings. \medskip }
\end{quote}

\vspace{4ex}

\section{Introduction} \label{sectionintro}

Classical Jordan H\"older theorems in group theory or in representation
theory assert that under some finiteness assumptions a given group, module
or representation has a finite composition series with simple factors and
that the simple factors are unique up to ordering and isomorphism. This
article proves a new kind of Jordan H\"older theorem, for derived categories
of hereditary artin algebras, that is for certain triangulated
categories.

\medskip

The ingredients of a Jordan H\"older theorem are the terms {\em composition
series} and {\em simple} objects. A finite composition series is a succession
of short exact sequences. A simple object is not allowed to
be the middle term of any non-trivial
short exact sequence. We propose to view recollements of triangulated
categories as analogues of short exact sequences. Hence iterated recollements,
frequently called {\em stratifications}, are analogues of composition series.
It then may seem natural to call a triangulated category simple if it does not
admit a non-trivial recollement by triangulated categories. We will see,
however, that the choice of definition is a subtle point - like in geometry
it is crucial to decide which kind of subobjects or factor objects are to occur
in stratifications. For derived categories of rings
it turns out to be reasonable to call a derived category simple
if it does not admit a non-trivial recollement, whose factors again are
derived categories of rings. In this sense, the main result of this article
is:
\medskip

{\em {\bf Main Theorem \ref{maintheorem}.}
The (unbounded) derived module category of a hereditary artin
algebra admits a finite composition series, and the simple factors in a
composition series are unique up to ordering and equivalence of triangulated
categories.}
\medskip

The recollements of triangulated categories used here as analogues of exact
sequences describe the middle term by a triangulated subcategory and
a triangulated quotient category. Recollements have been first defined by
Beilinson, Bernstein and Deligne \cite{BBD} in geometric contexts, where
stratifications of spaces imply recollements of derived categories of sheaves,
by using derived versions of Grothendieck's six functors -
whose abstract properties in fact get axiomatized by
the notion of recollement. As certain derived categories of perverse sheaves
are equivalent to derived categories of modules over blocks of the
Bernstein-Gelfand-Gelfand category $\mathcal O$, recollements do exist for
the corresponding algebras as well. For these algebras,
the stratification provided by
iterated recollements, is by derived categories of vector spaces. This is
one of the fundamental, and motivating, properties of quasi-hereditary
algebras, introduced by Cline, Parshall and Scott \cite{CPS1}. Another
source of examples for stratifications of algebras in this sense is a result
of Beilinson \cite{Beilinson},
which identifies certain derived categories of sheaves with
derived categories of algebras; this relates, for instance the coherent
sheaves on a projective line with the path algebra of the Kronecker quiver.
Recently, recollements of derived categories also have come up in tilting
theory \cite{AKL}, in the context of tilting modules associated with
injective homological epimorphisms.

\medskip

The question of Jordan H\"older theorems being valid for (certain)
derived categories or more generally uniqueness of stratifications to hold
true for (certain) triangulated categories has come up about twenty years
ago with the work of Cline, Parshall and Scott \cite{CPS1,PS}. It has
motivated studies of examples by Wiedemann \cite{W} and Happel
\cite{Hap1} as well as the criterion for existence of recollements in \cite{K}.
Despite this interest, the main Theorem is the first positive result obtained so far for any class
of algebras and of derived categories. A similar result cannot hold true for all algebras: In general,
 both existence and uniqueness of a finite Jordan H\"older series of a derived category may fail. 
A counterexample to existence will be provided in Section \ref{sectionex}. The difficult problem of uniqueness has been 
taken up by Chen and Xi \cite{CX}, who have constructed an algebra, whose derived category has two different Jordan H\"older series of
 different lengths. They also provided examples of finite Jordan H\"older series of the same derived category having the same length, 
but different composition factors.

\medskip

Like exact sequences in abelian categories, recollements of triangulated
categories have - in general or under additional assumptions - associated
long exact sequences for various cohomology theories, such as
K-theory, cyclic homology, and Hochschild cohomology.
Moreover, the middle term and
the outer terms of a recollement of derived categories of rings share some
homological invariants; for instance, the middle term has finite global or
finitistic dimension if and only if the outer terms have so as well, see
\cite{Hap2}.
\medskip

This article is organised as follows. In Section \ref{sectionprel}
we recall the definitions needed and then we collect for later use
a variety of results from various backgrounds. In Section
\ref{sectioncomplete} we prove technical results on completing
recollement diagrams, which in our setup play the role of the
butterfly lemma used in proofs of the classical Jordan H\"older
theorems. In Section \ref{sectioncompare} we discuss lifting and
restricting of recollements between bounded or unbounded derived
categories, and we give some criteria and examples of derived simple
rings. Section \ref{sectionproof} contains the
proof of the main theorem \ref{maintheorem}. Section \ref{sectionex} provides various (counter)examples;
this illustrates in particular our choice of definition for
'derived simplicity'. 

\bigskip

{\bf Acknowledgements:} The second and the third named authors
would like to thank the University of Verona and
the colleagues there for the hospitality in 2009, when much of the
research reported in this article has been done.

The first named author acknowledges partial support from MIUR,
PRIN-2008 ``Anelli, algebre, moduli e categorie", and from
Progetto di Ateneo CPDA071244 of the University of Padova.

\bigskip
\section{Recollements, universal extensions and perpendicular
categories} \label{sectionprel}

In this section we recall the definition of recollements and then collect
information on connections with other concepts, which will be used
in the proof of the Main Theorem \ref{maintheorem}.

\smallskip
Throughout this article, rings or algebras are assumed to be associative
with a unit element. An artin algebra by definition is an artinian algebra
over a commutative artinian ring.

\subsection{Recollements.}
Let $\Xcal, \Ycal$ and $\Dcal$ be triangulated categories. $\Dcal$
is said to be  a {\em recollement} (\cite{BBD}, see also
\cite{PS}) of $\Xcal$ and $\Ycal$ if there are six triangle
functors as in the following diagram
\[
\xy (-50,0)*{\mathcal Y}; {\ar (-27,3)*{}; (-43,3)*{}_{i^{\ast}}};
{\ar (-43,0)*{}; (-27,0)*{}|{ i_{\ast}=i_!}}; {\ar(-27,-3)*{};
(-43,-3)*{}^{i^!}}; (-20,0)*{\mathcal D}; {\ar (3,3)*{};
(-13,3)*{}_{j_!}}; {\ar (-13,0)*{}; (3,0)*{}|{j^!=j^\ast}};
{\ar(3,-3)*{}; (-13,-3)*{}^{j_{\ast}}}; (10,0)*{\mathcal X};
\endxy
\]
such that
\begin{enumerate}
\item $(i^\ast,i_\ast)$,\,$(i_!,i^!)$,\,$(j_!,j^!)$ ,\,$(j^\ast,j_\ast)$ are adjoint pairs;
\smallskip

\item $i_\ast,\,j_\ast,\,j_!$  are full embeddings;
\smallskip

\item  $i^!\circ j_\ast=0$ (and thus also $j^!\circ i_!=0$ and $i^\ast\circ j_!=0$);
\smallskip

\item  for each $C\in \Dcal$ there are triangles $$i_! i^!(C)\to
C\to j_\ast j^\ast (C)\leadsto $$
 $$j_! j^! (C)\to C\to i_\ast i^\ast(C)\leadsto$$
where the four morphisms staring from/ending at $C$ are the unit/counits of the adjoint pairs in (1).
\end{enumerate}

In this paper, $\Dcal$ will always be a derived module category
$\Dcal(R)=\Dcal(\text{Mod-}R)$ of some ring $R$ (with unit). By Mod-$R$ we mean the
category of right $R$-modules. Later on we will work with hereditary
artin algebras.

\medskip
\subsection{Homological epimorphisms and universal localization}

Let $\lambda\colon R\rightarrow S$ be a ring epimorphism, that is,
an epimorphism in the category of rings.
 Following  Geigle and Lenzing \cite{GL}, we say that
   $\lambda$ is  a  \emph{homological ring
epimorphism} if $\mbox{Tor} _i^R(S,S)=0$ for all $i>0$. Note that
this holds true if and only if the restriction functor
$\lambda_\ast:\Dcal(S)\to\Dcal(R)$ induced  by $\lambda$ is fully
faithful (\cite[4.4]{GL}, \cite[5.3.1]{N}). The following result
connects homological epimorphisms and
recollements, where their epiclasses and equivalence classes are
defined naturally. For more details see \cite{AKL}.

\medskip
\begin{prop} (\cite{AKL}, 1.7)
\label{chooseBwithhomepi} There is a bijection between the
epiclasses of homological ring epimorphisms and the equivalence
classes of those recollements  
\[
\xy (-39,0)*{\mathcal Y}; {\ar (-22,2)*{}; (-32,2)*{}_{i^{\ast}}};
{\ar (-32,0)*{}; (-22,0)*{}}; {\ar(-22,-2)*{};
(-32,-2)*{}}; (-15,0)*{\mathcal D}; {\ar (2,2)*{}; (-8,2)*{}};
{\ar (-8,0)*{}; (2,0)*{}}; {\ar(2,-2)*{}; (-8,-2)*{}};
(9,0)*{\mathcal X};
\endxy
\]

\vspace{0.3cm} for which $\Dcal=\Dcal(R)$ for some ring $R$ and
$i^\ast(R)$ is an exceptional object of $\Ycal$.
\end{prop}
Given such a recollement, the homological ring epimorphism is given by 
$$R=\End(R) \xrightarrow{\lambda} \End(i^*(R))=:S.$$
Up to equivalence, $\Ycal=\Dcal(S)$, $i^*=-\lten_R S$, $i_*=\lambda_*$ and $i^!=\mathbf{R}\Hom_R(S,-)$.


\medskip
{\bf Theorem.} {\cite[Theorem~4.1]{Schofieldbook}}
\label{def:universallocalization} {\it Let $R$ be a ring and $\Sigma$ be a set of
morphisms between finitely generated projective right $R$-modules.
Then there exist a ring $R_\Sigma$ and a morphism of rings
$\lambda\colon R\rightarrow R_\Sigma$ such that
\begin{enumerate}
\item $\lambda$ is \emph{$\Sigma$-inverting,} i.e. if
$\alpha\colon P\rightarrow Q$ belongs to  $\Sigma$, then
$\alpha\otimes_R 1_{R_\Sigma}\colon P\otimes_R R_\Sigma\rightarrow
Q\otimes_R R_\Sigma$ is an isomorphism of right
$R_\Sigma$-modules, and
\item $\lambda$ is \emph{universal
$\Sigma$-inverting}, i.e. if $S$ is a ring such that there exists
a $\Sigma$-inverting morphism $\psi\colon R\rightarrow S$, then
there exists a unique morphism of rings $\bar{\psi}\colon
R_\Sigma\rightarrow S$ such that $\bar{\psi}\lambda=\psi$.
\end{enumerate}
}

\medskip

The morphism $\lambda\colon R\rightarrow R_\Sigma$ is a ring
epimorphism with  $\Tor^R_1(R_\Sigma,R_\Sigma)=0.$ It
 is called the \emph{universal localization of $R$ at
$\Sigma$}.

\smallskip

In general, a universal localization  need not be a homological
ring epimorphism, see \cite{NRS} (and also \cite[Example 5.4]{AKL} for a
different kind of example). For a hereditary ring
 $R$, however, $\lambda\colon R\rightarrow R_\Sigma$ is always a homological epimorphism, and $R_\Sigma$ is a hereditary ring.
In fact, it is even shown in \cite[6.1]{KS} that over hereditary rings universal localizations  coincide with
 homological epimorphisms.

\medskip

Let now $\mathcal{U}$ be a set of  finitely presented right
$R$-modules of projective dimension one. For each
$U\in\mathcal{U},$ consider a morphism $\alpha_U$ between finitely
generated projective right $R$-modules such that
$$0\to P\stackrel{{\alpha_U}}{\to} Q\to U\to 0$$
 We will denote by
$R_{\mathcal{U}}$ the universal localization of $R$ at
$\Sigma=\{\alpha_U\mid U\in\mathcal{U}\}.$ In fact,   $R_{\mathcal
U}$ does not depend on the class $\Sigma$  chosen,
cf.~\cite[Theorem~0.6.2]{Cohnfreerings}, and we will also call it
the \emph{universal localization of $R$ at ${\mathcal{U}}$}.





\medskip

\subsection{Classical tilting modules}

Suppose $R$ is a ring. Recall that an $R$-module
 $T$ is said to be a {\em  tilting module (of
projective dimension at most one)}\label{tilt} if the following
conditions are satisfied:
\begin{enumerate}

\item  proj.dim$(T) \le 1$;

\item  $\Ext^1_R(T,T^{(I)}) = 0$ for each set $I$; and

\item  there is an  exact sequence $0 \to R \to T_0 \to T_1
\to 0$ where  $T_0,T_1$ belong to Add$T$. \end{enumerate} The module
$T$ is called a {\em partial tilting module} if it satisfies the
conditions (1) and (2).
If, in addition, $T$ is finitely presented, then we say that $T$ is a  {\em classical}  (partial) tilting module.

\smallskip

It was shown in \cite{AKL} that classical partial tilting modules
induce recollements. In Theorem \ref{HRSextended} below we  state this result  for  the case of a hereditary ring, where there is an important
connection to universal localizations.

\medskip

\subsection{Exceptional objects.} \label{exceptional}
Let us  turn to the derived  category $\mathcal D=\mathcal D(R)$. Recall that
$X\in \mathcal D$ is {\it exceptional} if $\Hom_{\mathcal D}(X,X[n])=0$ for all non-zero integers $n$.
Further, the analog of finitely presented modules is provided by  the \emph{compact} objects, that is, the objects $X\in \mathcal D$ such that  the functor
$\Hom_{\mathcal D}(X,-)$ preserves small coproducts, or equivalently,  $X$ is quasi-isomorphic to a bounded complex consisting of finitely generated projective modules.
Of course, a finitely presented $R$-module over a hereditary ring $R$  is exceptional if and only if it is a classical partial tilting module.

\medskip

Let now $X\in \mathcal D$ be  a compact exceptional object and denote by $\Tria X$  the smallest full triangulated subcategory of $\mathcal D$ which contains $X$ and
is closed under small coproducts. If $\Tria X= \mathcal D$, then $X$ is said to be a {\it tilting complex}. In general, we know from
 \cite{Ke} that $\Tria X$ is equivalent to the derived category $\Dcal(C)$ of
  the endomorphism ring $C=\End_{\mathcal D}(X)$ of $X$.

\smallskip

The following result characterizes the existence of a recollement of $\mathcal D$ by derived categories of rings in terms of  a suitable pair of exceptional objects.

\begin{thm}(\cite{K}, \cite[5.2.9]{N}, \cite{NS1})\label{single} There are rings $R,B,C$ with  a  recollement of the form
\[
\xy (-45,0)*{\Dcal(B)};
{\ar (-25,2)*{}; (-35,2)*{}}; {\ar (-35,0)*{}; (-25,0)*{}};
{\ar(-25,-2)*{}; (-35,-2)*{}}; (-15,0)*{\Dcal(R)};
{\ar (5,2)*{}; (-5,2)*{}_{j_!}}; {\ar (-5,0)*{}; (5,0)*{}}; {\ar
(5,-2)*{}; (-5,-2)*{}}; (15,0)*{\Dcal(C)};
\endxy
\]

\vspace{0.3cm}
if and only if there are exceptional objects $X,Y\in \Dcal(R)$ such that
\begin{enumerate}
\item[(i)]  $X$ is  compact,
\item[(ii)] $Y$ is a self-compact object, that is, $\Hom_{\mathcal D}(Y,-)$ preserves small coproducts in $\Tria Y$,
\item[(iii)] $\Hom_{\Dcal}(X[n],Y)=0$ for all $n\in\Z$,
\item[(iv)] $X\oplus Y$ generates $\Dcal(R)$, that is, an object $C\in \Dcal(R)$ is   zero whenever $\Hom_{\Dcal} (X\oplus Y,C[n])=  0$ for every integer $n$.
\end{enumerate}
In particular,
  $X=j_!(C)$, and  $\Tria X$ is equivalent to $\Dcal(C)$.\end{thm}

\medskip

\subsection{Exceptional sequences.} \label{exceptseq}
In this  subsection, let $A$ be a hereditary artin algebra
with $n$ non-isomorphic simple modules.
Recall that a sequence of
exceptional $A$-modules $(X_1, X_2, \ldots, X_m)$ is called an {\it
exceptional sequence} if $\Hom_A(X_j,X_i)=0$ and $\Ext^1_A(X_j,X_i)=0$ for
each pair $i<j$.
 An exceptional sequence $(X_1, X_2, \ldots, X_m)$
is called {\it complete} if $m=n$.

\medskip

By Schur's lemma the endomorphism ring of a simple module is a
skew-field. This is also the case for any indecomposable
exceptional module  of finite length by a result of Happel and Ringel.

\begin{prop} {\cite[4.1 and 4.2]{HR}} \label{Endofexceptional}
(1) If $X_A$ is an indecomposable, finitely generated, and
exceptional module, then the endomorphism ring of $X$ is a skew-field.


(2) If $X_A$ is a finitely generated,
multiplicity-free, and exceptional $A$-module, then its  indecomposable direct summands  can be arranged into an
exceptional sequence, which will be complete whenever $X$ is a classical tilting $A$-module.
\end{prop}

\begin{thm} {\cite[Theorem~4]{R2}}\label{endotilting}
Let $A$ be a hereditary artin
algebra, and $T$ a multiplicity-free classical tilting $A$-module.
Then the endomorphism rings of the indecomposable summands of $T$
are precisely the endomorphism rings of the non-isomorphic simple
modules.
\end{thm}

\medskip

The concept of  (complete) exceptional sequence is available in
the derived  category $\Dcal(A)$  as well.
Since $A$ is hereditary,  the indecomposable objects in $\Dcal(A)$ are the shifts
of the indecomposable $A$-modules. Hence Proposition  \ref{Endofexceptional} holds in  $\Dcal(A)$, too.
Given a  compact, multiplicity-free, and exceptional object $X$ in $\Dcal(A)$, we can decompose $X$ into
a direct sum $Y_1[k_1]\oplus Y_2[k_2]\oplus \ldots \oplus Y_s[k_s]$
such that the $Y_i$'s are $A$-modules and $k_1 < k_2 <\ldots < k_s$.
Since modules have no extensions in negative degrees,
there are no nontrivial homomorphisms from $Y_i[k_i]$ to
$Y_j[k_j]$ whenever $i>j$. Hence we can order the indecomposable
direct summands of $X$ into an exceptional sequence, which will be complete whenever  $X$
generates $\Dcal(A)$ and therefore has $n$ indecomposable
direct summands.

Here the question arises, whether Theorem \ref{endotilting}
holds for tilting complexes. This question will be answered positively
later, in Corollary \ref{endoftiltingcomplex}.

\medskip
\subsection{Perpendicular categories}\label{perpendicular}

Recollements are closely related to torsion theories and the outer terms
in a recollement are equivalent to certain perpendicular categories, see
\cite{K,N}. Perpendicular categories behave especially well in hereditary
situations.

For any ring $R$ and any $R$-module $X$, the {\it perpendicular category}
$\widehat{X}$ is by definition the full subcategory of Mod-$R$
consisting of the modules $M$ satisfying $\Hom_R(X,M)=0$ and
$\Ext^1_R(X,M)=0$.

\medskip

The next result extends a result by Happel, Rickard and Schofield
from finite dimensional hereditary algebras to semihereditary
rings, and from module category level to derived category level.
Recall that a ring is hereditary if all submodules of projective
modules are again projective. If this is required only for
finitely generated submodules, it is called {\em semihereditary}.
For example,  {\em
von Neumann regular} rings, that is, rings $R$ such that every element $a$ can be written as
 $a=axa$ for some $x$ in $R$ (depending on $a$), are semihereditary.
 Commutative
semihereditary domains are called {\em Pr\"ufer domains}. The subring of $\mathbb C$ consisting of the algebraic integers is a non-noetherian Pr\"ufer domain of global dimension 2, cf.~\cite[VI,4.5]{FS}. Another example of a ring that is semihereditary but not hereditary (on one side) can be found in \cite[2.33]{Lam}.

\medskip
\begin{thm} \label{HRSextended}
Let $R$ be a semihereditary ring, and $X_R$ a finitely presented,
exceptional  $R$-module. Then there exists a ring $B$ such that
the following holds:
\begin{enumerate}
\item (\cite[Proposition 3]{HRS}) The perpendicular category
$\widehat{X}$ is equivalent to $\mathrm{Mod}\text{-}B$.
\item There is a recollement \[
\xy (-45,0)*{\mathcal \Dcal(B)}; {\ar (-25,2)*{}; (-35,2)*{}};
{\ar (-35,0)*{}; (-25,0)*{}^{}}; {\ar(-25,-2)*{}; (-35,-2)*{}};
(-15,0)*{{\mathcal D}(R)}; {\ar (5,2)*{}; (-5,2)*{}_{}};
{\ar (-5,0)*{}; (5,0)*{}}; {\ar (5,-2)*{}; (-5,-2)*{}};
(15,0)*{{\mathcal \Dcal(C)}};
\endxy
\]
where $C=\End_R(X)$ is the endomorphism ring of $X$.
\item The ring $B$ can be chosen as universal localization of $R$ at $X$. Further, $B$ is hereditary if so is $R$.
\end{enumerate}
\end{thm}

\medskip
\begin{proof}
Since $R$ is semihereditary, every finitely presented $R$-module
has projective dimension $\le 1$. The statements are thus
contained in \cite[Example 4.5 and Theorem 4.8]{AKL}, we give some details for the
reader's convenience. By \cite[Lemma 4.1]{AKL} the perpendicular category $\widehat{X}$  coincides with  the essential image of the
restriction functor $\lambda_*$ given by the universal
localization $\lambda:R\to B$ at $X$. Moreover, $\widehat{X}$ is a reflective subcategory of $\M$. This means by definition
that every module $M\in\M$ admits a $\widehat{X}$-reflection, that is, a morphism $\eta_M:M\ra N$ such that $N\in\widehat{X}$ and 
$\Hom_R(\eta_M,Y):\Hom_R(N,Y)\ra\Hom_R(M,Y)$ is bijective for all $Y\in\widehat{X}$. 

\smallskip

{\it First case:} $X$ is projective. 
 $\lambda:R\to B$ can be chosen as universal
localization  at the zero map
$\Sigma=\{\sigma:0\to X\}$. Hence it is a homological epimorphism, because
$R$ has
weak global dimension bounded by 1, see \cite[4.67]{Lam}. 
Then  $\Dcal(R)$ is a recollement of $\Tria X\cong \Dcal(C)$ and
$\Dcal(B)$, see \cite[Example 4.5]{AKL}.

Note also that the $R$-module $B_R$
is isomorphic to  the  $\widehat{X}$-reflection of $R$, and by \cite[Section 1]{CTT} the latter coincides with $R/\tau_X(R)$ where $\tau_X$ denotes the trace of $X$.

\smallskip

{\it Second case:} $X$ has projective dimension one. The
 universal
localization $\lambda:R\to B$ at $X$ is a homological
epimorphism by \cite[4.67]{Lam}. 
Then $\Dcal(R)$ is a recollement of $\Tria X\cong \Dcal(C)$ and
$\Dcal(B)$ by \cite{NR}.
\smallskip

For later application, we give an explicit description of $B$ as  Bongartz complement of $X$, cf.~\cite[Section 1]{CTT}: if $c$ is the minimal number of
generators of $\Ext^{1}_{R}(X,R)$ as a module over $C=\End_R(X)$, then there exists an exact sequence
$$E: 0\to{R}\to{M}\to{X^{(c)}}\to{0}$$
with the following properties:
\\(1) any exact sequence $0\to{R}\to{N}\to{X}\to{0}$ has the form $Ef$ for some $f\in \Hom_R(X,X^{(c)})$,
 \\(2) $T=M\oplus X$ is a tilting module,
 \\(3) $M/\tau_X(M)$ is the $\widehat{X}$ -reflection of $R$,
 \\(4) $B\cong \End_R(M)/\tau_X(M)$ as rings, and $B\cong M/\tau_X(M)$ as $R$-modules.
 \end{proof}

\smallskip

\begin{prop} \label{artin} In the situation of Theorem \ref{HRSextended}, assume in addition that $R$ is an artin algebra and $X$ is indecomposable. Then the following hold true.

(1)  If $X$ is projective, then  $B$  is an artin algebra, and the simple $B$-modules are precisely the simple $R$-modules that are not isomorphic to $X/\mathrm{Rad}(X)$.

(2) If $X$ has projective dimension  one, then  $B$  is an artin algebra, and
 viewed as an $R$-module, $B$  complements $X$ to a tilting module $T=B\oplus X$.
 \end{prop}
 \begin{proof}
 (1) Let $e$ be an idempotent such that $X=eR$. Then $B\cong R/ReR$ is an artin algebra. Moreover, $\widehat{X}=\{Y_R\,\mid\, Ye=0\}$ is closed under submodules, so every simple $B$-module is also a simple $R$-module, and of course it is not isomorphic to $eR/e\,\mathrm{Rad}(R)$. The converse implication is obvious.

(2) If $X$ is indecomposable, then $C=\End_R(X)$ is a skew-field by Proposition \ref{Endofexceptional}, and  $c$ is  the $C$-dimension of $\Ext^1_R(X,R)$, which is
finite because $X$ is finitely generated. Applying
$\Hom_R(X,-)$ to the universal
sequence $0 \rightarrow R \rightarrow M \rightarrow X^{c}
\rightarrow 0$, we obtain a long exact
sequence $$0 \ra \Hom_R(X,R) \ra \Hom_R(X,M) \ra \Hom_R(X,X^c) \ra
\Ext^1_R(X,R)$$ $$\ra \Ext^1_R(X,M) \ra \Ext^1_R(X,X^c) \ra 0$$ (recall
that $R$ is hereditary). The map $\Hom_R(X,X^c) \ra \Ext^1_R(X,R)$ is
bijective by construction,
$\Hom_R(X,R)=0$ because $X$ is not projective, and
$\Ext^1_R(X,X^c)=0$ since $X$ is exceptional. Therefore $M\in\widehat{X}$ is the
$\widehat{X}$-reflection of $R$. Hence, as an $R$-module, $B\cong M$ is the Bongartz complement of $X$, and moreover,  $B\cong \End_R(M)$ is an artin algebra because $M$ is finitely generated.
\end{proof}
\bigskip

\section{Completing recollement diagrams}
\label{sectioncomplete}

Proofs of classical Jordan H\"older theorems typically employ an argument
called butterfly lemma, which helps to compare composition series of various
(sub)objects. The results of this Section will serve a similar purpose for
triangulated or derived categories.

\medskip
\begin{prop} \label{completingrecoll}
Let $A$ be a ring.
Every diagram of the following form, involving a horizontal recollement and
a vertical one,
\[
\xy (-40,0)*{\Dcal(B)}; {\ar (-20,3)*{}; (-30,3)*{}_a}; {\ar
(-30,0)*{}; (-20,0)*{}|{\,b\,}}; {\ar(-20,-3)*{}; (-30,-3)*{}^c};
(-10,0)*{\Dcal(A)}; {\ar (10,3)*{}; (0,3)*{}_d}; {\ar (0,0)*{};
(10,0)*{}|{\,e\,}}; {\ar (10,-3)*{}; (0,-3)*{}^f};
(20,0)*{\Dcal(C)}; {\ar (23,5)*{}; (23,15)*{}_i}; {\ar (20,15)*{};
(20,5)*{}|{\begin{array}{c}\scriptstyle{h}\end{array}}}; {\ar
(17,5)*{}; (17,15)*{}^g}; (20,20)*{\Dcal(E)}; {\ar (17,-15)*{};
(17,-5)*{}^k}; {\ar (20,-5)*{};
(20,-15)*{}|{\begin{array}{c}\scriptstyle{l}\end{array}}}; {\ar
(23,-15)*{}; (23,-5)*{}_m}; (20,-20)*{\Dcal(F)};
\endxy
\]
can be completed to a diagram of the following form, involving two
horizontal and two vertical recollements,
\[
\xy (-60,0)*{\Dcal(B)};
{\ar (-35,3)*{}; (-45,3)*{}_a}; {\ar (-45,0)*{};
(-35,0)*{}|{\,b\,}}; {\ar(-35,-3)*{}; (-45,-3)*{}^c};
(-20,0)*{\Dcal(A)}; {\ar (-17,-15)*{}; (-17,-5)*{}_{md}}; {\ar
(-20,-5)*{};
(-20,-15)*{}|{\begin{array}{c}\scriptstyle{el}\end{array}}}; {\ar
(-23,-15)*{}; (-23,-5)*{}^{kf}}; (-20,-20)*{\Dcal(F)};
{\ar (5,3)*{}; (-5,3)*{}_d}; {\ar (-5,0)*{}; (5,0)*{}|{\,e\,}}; {\ar
(5,-3)*{}; (-5,-3)*{}^f}; {\ar (5,-23)*{}; (-5,-23)*{}^1}; {\ar
(-5,-20)*{}; (5,-20)*{}|{\,1\,}}; {\ar (5,-17)*{}; (-5,-17)*{}_1};
(20,0)*{\Dcal(C)}; {\ar (5,23)*{}; (-5,23)*{}_r}; {\ar (-5,20)*{};
(5,20)*{}|{\,s\,}}; {\ar (5,17)*{}; (-5,17)*{}^t};
(-60,20)*{\Dcal(B)};
{\ar (-35,23)*{}; (-45,23)*{}_u}; {\ar (-45,20)*{};
(-35,20)*{}|{\,v\,}}; {\ar(-35,17)*{}; (-45,17)*{}^w};
(-20,20)*{\Dcal(G)};
{\ar (23,5)*{}; (23,15)*{}_i}; {\ar (20,15)*{};
(20,5)*{}|{\begin{array}{c}\scriptstyle{h}\end{array}}}; {\ar
(17,5)*{}; (17,15)*{}^g}; (20,20)*{\Dcal(E)};
{\ar (17,-15)*{}; (17,-5)*{}^k}; {\ar (20,-5)*{};
(20,-15)*{}|{\begin{array}{c}\scriptstyle{l}\end{array}}}; {\ar
(23,-15)*{}; (23,-5)*{}_m}; (20,-20)*{\Dcal(F)}; {\ar (-57,5)*{};
(-57,15)*{}_1}; {\ar (-60,15)*{};
(-60,5)*{}|{\begin{array}{c}\scriptstyle{1}\end{array}}}; {\ar
(-63,5)*{}; (-63,15)*{}^1}; {\ar (-17,5)*{}; (-17,15)*{}_z}; {\ar
(-20,15)*{};
(-20,5)*{}|{\begin{array}{c}\scriptstyle{y}\end{array}}}; {\ar
(-23,5)*{}; (-23,15)*{}^x};
\endxy
\]
where $G$ is a differential graded algebra.
\end{prop}


\medskip
\begin{proof}
First we fill in the bottom square in the diagram by putting all
functors from $\Dcal(F)$ to itself to be identity, and by using the
compositions $m \cdot d$, $e \cdot l$ and $k \cdot f$ to connect
$\Dcal(A)$ and $\Dcal(F)$. Next we can complete by \cite[Theorem
2.4(b)]{PS} the half-recollement involving $\Dcal(A)$ and $\Dcal(F)$
to get some triangulated category in the middle of the top row. By
\cite[4.3.6, 4.4.8]{N}, $z(A)$ is a compact generator of this
triangulated category, and hence by \cite[Theorem 4.3]{K} this is
equivalent to the derived category of the differential graded
endomorphism algebra $G$ of $z(A)$.

\medskip

Now we complete the upper left square: The embedding functor $v$
exists, since the composition $1_B \cdot (b \cdot e) \cdot l$
vanishes and thus the image of $b$ must be inside $\Dcal(G)$. The
right adjoint $w$ is defined by $w:= y \cdot c \cdot 1_B$, and the
composition satisfies $v \cdot w = v \cdot y \cdot c \cdot 1_B = 1_B
\cdot b \cdot c \cdot 1_B = 1_B$. Similarly, the left adjoint $u$ is
defined by $u:= y \cdot a \cdot 1_B$ and composition satisfies $v
\cdot u = v \cdot y \cdot a \cdot 1_B = 1_B \cdot b \cdot a \cdot
1_B = 1_B$.

\medskip

To complete the upper right square, we define $s$ be $s:= y \cdot e
\cdot g (= y \cdot e \cdot i)$. Then $r:= h \cdot d \cdot z$ is a
left adjoint: The image of of $h \cdot d$ is contained in the kernel
of $e \cdot l$ and thus in the image of the embedding $y$.
Therefore, $r \cdot s = h \cdot d \cdot z \cdot y \cdot e \cdot g =
h \cdot d \cdot e \cdot g = h \cdot g = 1_E$. Similarly, $t:= h
\cdot f \cdot x$ is a right adjoint: Again, the image of $h \cdot f$
is contained in the image of the embedding $y$, and therefore $t
\cdot s = h \cdot f \cdot x \cdot y \cdot e \cdot g = h \cdot f
\cdot e \cdot g = h \cdot g = 1_E$.

\medskip

Finally, we check that the top row really forms a recollement. By
definition the functors $v$, $r$ and $t$ are full embeddings. The
composition $v \cdot s$ vanishes, since $v \cdot s  = v \cdot y \cdot e \cdot g = 1_B
\cdot b \cdot e\cdot g = 0$. The kernel of $s$ is indeed the kernel of
$y\cdot e$, so it is precisely $\Dcal(B)$. It remains to check the
existence of the canonical triangles. Let $X$ be in an object in
$\Dcal(G)$ and write it as middle term of a canonical triangle (for
the given recollement in the second row) $Y \rightarrow X
\rightarrow Z \leadsto$, where $Y$ is in $\Dcal(B)$ and $Z$ is in
$\Dcal(C)$. Since $X$ is in $\Dcal(G)$, its $e$-image $Z$ is in the
kernel of $l$ and thus it is in $\Dcal(E)$, as required. Thus the
given triangle also is a canonical triangle for the first row. The
second canonical triangle for the second row, for the given object
$X$, is $U \rightarrow X \rightarrow V \leadsto$ with $U$ in
$\Dcal(C)$ and $V$ in $\Dcal(B)$. This triangle serves as a
canonical triangle for the first row as well, once we have shown
that $U$ is already in $\Dcal(E)$: The object $U$ is obtained from
$X$ by applying $y \cdot e$, and thus the image of $U$ under $l$ is
the image of $X$ under $y \cdot e \cdot l = s \cdot h \cdot l = 0$.
\end{proof}

\medskip

Dually one can prove the following.

\begin{prop} \label{completingrecoll-dual}
Let $A$ be an algebra. Every diagram of the following form,
involving a horizontal recollement and a vertical one,
\[
\xy (-60,0)*{\Dcal(B)}; {\ar (-57,5)*{}; (-57,15)*{}_i}; {\ar
(-60,15)*{};
(-60,5)*{}|{\begin{array}{c}\scriptstyle{h}\end{array}}}; {\ar
(-63,5)*{}; (-63,15)*{}^g}; (-60,20)*{\Dcal(E)}; {\ar (-63,-15)*{};
(-63,-5)*{}^k}; {\ar (-60,-5)*{};
(-60,-15)*{}|{\begin{array}{c}\scriptstyle{l}\end{array}}}; {\ar
(-57,-15)*{}; (-57,-5)*{}_m}; (-60,-20)*{\Dcal(F)}; {\ar (-40,3)*{};
(-50,3)*{}_a}; {\ar (-50,0)*{}; (-40,0)*{}|{\,b\,}};
{\ar(-40,-3)*{}; (-50,-3)*{}^c}; (-30,0)*{\Dcal(A)}; {\ar
(-10,3)*{}; (-20,3)*{}_d}; {\ar (-20,0)*{}; (-10,0)*{}|{\,e\,}};
{\ar (-10,-3)*{}; (-20,-3)*{}^f}; (0,0)*{\Dcal(C)};
\endxy
\]
can be completed to a diagram of the following form, involving two
horizontal and two vertical recollements,
\[
\xy (-60,0)*{\Dcal(B)};
{\ar (-35,3)*{}; (-45,3)*{}_a}; {\ar (-45,0)*{};
(-35,0)*{}|{\,b\,}}; {\ar(-35,-3)*{}; (-45,-3)*{}^c};
(-20,0)*{\Dcal(A)}; {\ar (-17,-15)*{}; (-17,-5)*{}}; {\ar
(-20,-5)*{}; (-20,-15)*{}}; {\ar (-23,-15)*{}; (-23,-5)*{}};
(-20,-20)*{\mathcal{X}};
{\ar (5,3)*{}; (-5,3)*{}_d}; {\ar (-5,0)*{}; (5,0)*{}|{\,e\,}}; {\ar
(5,-3)*{}; (-5,-3)*{}^f}; {\ar (5,-23)*{}; (-5,-23)*{}}; {\ar
(-5,-20)*{}; (5,-20)*{}}; {\ar (5,-17)*{}; (-5,-17)*{}};
(20,0)*{\Dcal(C)}; {\ar (-57,5)*{}; (-57,15)*{}_i}; {\ar
(-60,15)*{};
(-60,5)*{}|{\begin{array}{c}\scriptstyle{h}\end{array}}}; {\ar
(-63,5)*{}; (-63,15)*{}^g}; (-60,20)*{\Dcal(E)}; {\ar (-63,-15)*{};
(-63,-5)*{}^k}; {\ar (-60,-5)*{};
(-60,-15)*{}|{\begin{array}{c}\scriptstyle{l}\end{array}}}; {\ar
(-57,-15)*{}; (-57,-5)*{}_m}; (-60,-20)*{\Dcal(F)}; {\ar
(-35,23)*{}; (-45,23)*{}_1}; {\ar (-45,20)*{}; (-35,20)*{}|{\,1\,}};
{\ar (-35,17)*{}; (-45,17)*{}^1}; (-20,20)*{\Dcal(E)}; {\ar
(17,-15)*{}; (17,-5)*{}^1}; {\ar (20,-5)*{};
(20,-15)*{}|{\begin{array}{c}\scriptstyle{1}\end{array}}}; {\ar
(23,-15)*{}; (23,-5)*{}_1}; (20,-20)*{\Dcal(C)}; {\ar (-17,5)*{};
(-17,15)*{}_{ai}}; {\ar (-20,15)*{};
(-20,5)*{}|{\begin{array}{c}\scriptstyle{hb}\end{array}}}; {\ar
(-23,5)*{}; (-23,15)*{}^{cg}}; {\ar (-35,-17)*{}; (-45,-17)*{}};
{\ar (-45,-20)*{}; (-35,-20)*{}}; {\ar(-35,-23)*{}; (-45,-23)*{}};
\endxy
\]
where $\mathcal{X}$ is a triangulated category.
\end{prop}

It follows from the proposition that the triangulated category
$\mathcal{X}$ has a recollement structure filtered by the two
derived module categories $\Dcal(F)$ and $\Dcal(C)$. But in general
we don't know whether $\mathcal{X}$ itself is equivalent to a
derived module category.

\medskip
For hereditary artin algebras, we are now able to generalize  Theorem \ref{HRSextended} to an
 exceptional and compact complex $X$.

\begin{cor} \label{HRSextended-more}
Let $A$ be a hereditary artin algebra, and $X$
an exceptional and compact complex in $\Dcal(A)$. Then there exists
a hereditary artin algebra $B$ with a homological ring epimorphism $A\to B$ and a recollement
\[
\xy (-45,0)*{\mathcal \Dcal(B)}; {\ar (-25,2)*{}; (-35,2)*{}};
{\ar (-35,0)*{}; (-25,0)*{}^{}}; {\ar(-25,-2)*{}; (-35,-2)*{}};
(-15,0)*{{\mathcal D}(A)}; {\ar (5,2)*{}; (-5,2)*{}_{}};
{\ar (-5,0)*{}; (5,0)*{}}; {\ar (5,-2)*{}; (-5,-2)*{}};
(15,0)*{{\mathcal \Dcal(C)}};
\endxy
\]

where $C=\End_A(X)$ is the endomorphism ring of $X$.
\end{cor}

\begin{proof} Assume $X$ is multiplicity free. By Subsection \ref{exceptseq},
the indecomposable direct summands of $X$ can be ordered into
an exceptional sequence, say $(X_1, X_2, \ldots, X_s)$, in
$\Dcal(A)$. For each pair $i<j$, there is no nontrivial homomorphism
from $X_j$ to $X_i$. Each $X_i$ is a shift of an indecomposable, finitely presented,
exceptional  $A$-module, so we apply Theorem \ref{HRSextended} and Proposition \ref{artin}
 iteratively to $X_s$, $X_{s-1}$,
$\ldots$, and $X_1$. In the first step we obtain a hereditary artin
algebra $B_s$ such that we have a recollement
\vspace{0.1cm}
\[
\xy (-47,0)*{\Dcal(B_s)}; {\ar (-25,2)*{}; (-35,2)*{}};
{\ar (-35,0)*{}; (-25,0)*{}^{}}; {\ar(-25,-2)*{}; (-35,-2)*{}};
(-15,0)*{{\mathcal D}(A)}; {\ar (5,2)*{}; (-5,2)*{}_{}};
{\ar (-5,0)*{}; (5,0)*{}}; {\ar (5,-2)*{}; (-5,-2)*{}};
(17,0)*{\Dcal(C_s)};
\endxy
\]

\vspace{0.2cm}
where $C_s=\End_A(X_s)$ is the endomorphism ring of $X_s$, and $B_s$ is a universal localization of $A$. Now $X_1$,
$\ldots$, $X_{s-1}$ belong to $\Dcal(B_s)$. Applying Theorem
\ref{HRSextended} to $X_{s-1}$ and $B_s$, we obtain a hereditary
artin algebra $B_{s-1}$ such that
\vspace{0.1cm}
 \[
\xy (-48,0)*{\Dcal(B_{s-1})}; {\ar (-25,2)*{}; (-35,2)*{}};
{\ar (-35,0)*{}; (-25,0)*{}^{}}; {\ar(-25,-2)*{}; (-35,-2)*{}};
(-15,0)*{{\mathcal D}(B_s)}; {\ar (5,2)*{}; (-5,2)*{}_{}};
{\ar (-5,0)*{}; (5,0)*{}}; {\ar (5,-2)*{}; (-5,-2)*{}};
(18,0)*{\Dcal(C_{s-1})};
\endxy
\]

\vspace{0.2cm}
where $C_{s-1}=\End_{B_s}(X_{s-1})=\End_A(X_{s-1})$ is the endomorphism ring of $X_{s-1}$, and $B_{s-1}$ is a universal localization of $B_s$.
Now we are in the situation of Proposition
\ref{completingrecoll-dual}. By completing the recollements we
obtain a new recollement of $\Dcal(A)$ filtered by $\Dcal(B_{s-1})$
and some triangulated category $\mcx$, which again admits a
recollement filtered by $\Dcal(C_{s-1})$ and $\Dcal(C_s)$. By
construction, $\Dcal(C_{i}) \cong \Tria X_i$ for $i=s, s-1$,
hence $\mcx \cong \Tria (X_s\oplus X_{s-1}) \cong \Dcal(\widetilde{C})$
for $\widetilde{C}=\End_A(X_s\oplus X_{s-1})$. Moreover, the composition $A\to B_s\to B_{s-1}$ is a homological epimorphism. To finish the proof we just
have to continue iteratively.
\end{proof}

\medskip

In the situation of Proposition \ref{completingrecoll},
the image of $F$ under the functor $md$ is always exceptional and
compact. Thanks to the corollary, if $A$ is a hereditary artin
algebra, we can choose $G$ to be also hereditary and artin. This
fact will be used later in the inductive proof of our main result Theorem
\ref{maintheorem}.

\bigskip
\section{Lifting and restricting recollements, derived simplicity}
\label{sectioncompare}

In the literature, various kinds of recollements are used for
different purposes; these involve bounded, left or right bounded
or unbounded derived categories, homotopy categories of
projectives. 
Although we are focussing on unbounded derived categories in the
main part of this article, we collect in this Section some
information on comparing recollements of different types. Roughly
speaking, lifting to 'larger' categories always is possible, while
restricting to 'smaller' categories is problematic. We do not
provide a final answer to the problem, whether the existence of a
recollement always implies the existence of another one that can
be restricted. Nor do we solve the question, which functors in an
existing recollement do restrict.

\medskip

Let $A$, $B$ and $C$ be any rings. Recall that Mod-$A$ denotes the
category of arbitrary right $A$-modules, and write Proj-$A$ for the
full subcategory of all projective modules. We ask for relations
between the following recollements:

$$\begin{array}{cccccccc}
(R0) & & & K^b(\text{Proj-}B) &\xy {\ar (-30,3)*{}; (-40,3)*{}};
{\ar (-40,1)*{}; (-30,1)*{}^{}}; {\ar(-30,-1)*{}; (-40,-1)*{}};
\endxy & K^b(\text{Proj-}A) & \xy {\ar (-30,3)*{}; (-40,3)*{}};
{\ar (-40,1)*{}; (-30,1)*{}^{}}; {\ar(-30,-1)*{}; (-40,-1)*{}};
\endxy & K^b(\text{Proj-}C) \vspace{0.4cm}\\

(R1) & & & \Dcal^b(B) & \xy {\ar (-30,3)*{}; (-40,3)*{}}; {\ar
(-40,1)*{}; (-30,1)*{}^{}}; {\ar(-30,-1)*{}; (-40,-1)*{}};
\endxy & \Dcal^b(A) & \xy {\ar (-30,3)*{}; (-40,3)*{}}; {\ar
(-40,1)*{}; (-30,1)*{}^{}}; {\ar(-30,-1)*{}; (-40,-1)*{}};
\endxy & \Dcal^b(C) \vspace{0.4cm} \\

(R2) & & & \Dcal^-(B) & \xy {\ar (-30,3)*{}; (-40,3)*{}}; {\ar
(-40,1)*{}; (-30,1)*{}^{}}; {\ar(-30,-1)*{}; (-40,-1)*{}};
\endxy & \Dcal^-(A) & \xy {\ar (-30,3)*{}; (-40,3)*{}}; {\ar
(-40,1)*{}; (-30,1)*{}^{}}; {\ar(-30,-1)*{}; (-40,-1)*{}};
\endxy & \Dcal^-(C) \vspace{0.4cm} \\

(R3) & & & \Dcal(B) & \xy {\ar (-30,3)*{}; (-40,3)*{}}; {\ar
(-40,1)*{}; (-30,1)*{}^{}}; {\ar(-30,-1)*{}; (-40,-1)*{}};
\endxy & \Dcal(A) & \xy {\ar (-30,3)*{}; (-40,3)*{}}; {\ar
(-40,1)*{}; (-30,1)*{}^{}}; {\ar(-30,-1)*{}; (-40,-1)*{}};
\endxy & \Dcal(C)\\
\end{array}$$

\medskip
\begin{lem} If the ring $A$ has a recollement of the form $(R1)$, then
it has a recollement of the form $(R2)$. The converse holds true
if $A$ has finite global dimension. \label{r1tor2}
\end{lem}

\begin{proof} See \cite[Corollary 6]{K}, \cite[Theorem 2]{NS2}: From the
recollement $(R1)$, the complexes $Y=i_*(B),\,X=j_!(C)\in
K^b(\text{Proj-}A)$ provide the required candidates for the existence of
a recollement of the form $(R2)$. Note that
$\Hom_{\Dcal(A)}(Y,Y[n]^{(I)})\cong \Hom_{\Dcal(B)}(B,B[n]^{(I)})=0$ for all
non-zero integers  $n$ and all sets $I$, because  $i_*$ is a full
embedding. For the converse, see \cite[Theorem 7]{K}.
\end{proof}


\medskip
The same argument as above also proves  the following lifting of
recollements. The second statement follows from \cite[Proposition
4]{K}.

\begin{lem} If the ring $A$ has a recollement of the form $(R0)$, then
it has a recollement of the form $(R2)$. The converse holds true
if $A$ has finite global dimension. \label{r0tor2}
\end{lem}

\medskip
\begin{lem} If the ring $A$ has a recollement of the form $(R2)$, then
it has a recollement of the form $(R3)$. \label{r2tor3}
\end{lem}

\begin{proof} Given a recollement of the form $(R2)$, we know from \cite[Theorem 1]{K}, \cite[Theorem 2]{NS2} that
$X=j_!(C)$ is  compact exceptional, $Y=i_*(B)$ is
self-compact and exceptional, and $X\oplus Y$
generates $\Dcal(A)$.
So, by Theorem \ref{single} there
exists a recollement of the form $(R3)$.
\end{proof}


\medskip
\begin{lem}
If the ring $A$ has a recollement of  the form (R3) and  $Y=i_*(B)$ belongs
to $K^b(\mathrm{Proj}\text{-}A)$, then $A$ has a  recollement of the form (R2).

In particular, if a  hereditary artin algebra $A$ has a recollement of  the form (R3), then it has a recollement of the form $(R2)$.
 \label{r3tor2}
\end{lem}

\begin{proof} Given a recollement of the form $(R3)$, we know from Theorem \ref{single} that $X=j_!(C)$ is
a compact exceptional object, and $Y=i_*(B)$ is self-compact. Then the statement follows from the criterion in
\cite[Theorem 1]{K}, \cite[Theorem 2]{NS2}, since we have as in the proof of Lemma \ref{r1tor2} that $\Hom_{\Dcal(A)}(Y,Y[n]^{(I)})=0$ for all non-zero integers $n$ and all sets $I$.

If $A$ is a hereditary artin algebra, we can assume by  Corollary \ref{HRSextended-more}  that the recollement of the form $(R3)$ is induced by a  homological ring epimorphism $\lambda:A\to B$ to a hereditary artin algebra $B$, and $i_*$ is the canonical embedding  $\Dcal(B)\to \Dcal(A)$. So  $Y=i_*(B)=B$ is an $A$-module and thus  belongs to
$K^b(\text{Proj-}A)$.
\end{proof}


\begin{cor} \label{hereditaryArtin}
 For a hereditary artin algebra $A$, the following assertions are equivalent:
\begin{enumerate}
 \item $A$ has a recollement of the form $(R0)$;
\item $A$ has a recollement of the form $(R1)$;
\item $A$ has a recollement of the form $(R2)$;
\item $A$ has a recollement of the form $(R3)$.
\end{enumerate}
\end{cor}

\begin{proof}
 Combine Lemmas 4.1 -- 4.4.
\end{proof}

\medskip
\begin{ex}\label{r2NOTtor1} The following example from \cite[Example 8]{K}
provides a recollement of the form $(R2)$ which does not restrict
to $\Dcal^b$-level. Indeed, by \cite[Proposition 4]{K}, a recollement
of the form $(R2)$ restricts to a recollement of the form $(R1)$
if and only if the functor $j_!:\Dcal^-(C)\ra \Dcal^-(A)$
restricts to $\Dcal^b$-level - a condition that fails here. 

\medskip
{\rm Let $A$ be the finite dimensional algebra over a field $k$
given by
\[
\begin{picture}(100,10)
\put(-20,2){$\cdot$} \put(-26,2){\footnotesize $1$}
\put(22,2){$\cdot$} \put(28,2){\footnotesize $2$}
\put(-14,7){\vector(1,0){33}} \put(0,8){\footnotesize $\alpha$}
\put(18,1){\vector(-1,0){33}} \put(0,-7){\footnotesize $\beta$}
\put(60,2){$[\beta\alpha\beta=0].$}
\end{picture}
\]
The simple module $S(1)$ and the projective module $P(2)$ provide
a recollement of $A$ of the form $(R2)$ with $B=\End_A(S(1))\cong k$
and $C=\End_A(P(2))\cong k[x]/(x^2)$. By construction, the functor
$j_!$ is given by the left derived functor $-\lten_C P(2)$. It
cannot restrict to $\Dcal^b$-level, for $P(2)$ as a left $C$-module
has infinite projective dimension.}
\end{ex}


\bigskip

For the rest of the section, we focus on
rings that do not admit non-trivial recollements as above. The following definition slightly extends a
definition of Wiedemann \cite{W}, who considered bounded derived
categories only.

\begin{defn}
A ring $R$ is called {\em derived simple} if $\Dcal(R)$ does not
admit any non-trivial recollement whose factors are derived
categories of rings.

A ring $R$ is called {\em derived simple} with respect to
$\Dcal^{\ast}$ if $\Dcal^{\ast}(R)$ for $\ast = \{b, +, - \}$ does not
admit any non-trivial recollement whose factors are derived
categories (of the form $\Dcal^{\ast}$) of rings.
\end{defn}

In Section \ref{sectionex} we will see why it is necessary to
require the factors to be derived categories of rings again, and
not just triangulated categories. Observe that  $\Dcal^-$-derived
simplicity implies $\Dcal^b$-derived simplicity by Lemma \ref{r1tor2},
and $\Dcal$-derived simplicity implies $\Dcal^-$-derived simplicity by
Lemma \ref{r2tor3}. As we will see in \cite{AKLY2}, the converse does not hold for general rings. However
the situation ist fine for hereditary artin algebras.

\smallskip

\begin{cor}
For a  hereditary artin algebra $A$, the following assertions are equivalent:
\begin{enumerate}
 \item $A$  is derived simple.
\item $A$ is $\Dcal^-$-derived simple. 
\item $A$ is $\Dcal^b$-derived simple.
\end{enumerate}
\end{cor}
\begin{proof} It follows from Corollary \ref{hereditaryArtin}. 
\end{proof}

\medskip

Derived simple rings obviously satisfy a Jordan H\"older theorem
for derived categories. Computations or proofs in K-theory or
about homological dimensions that are based on recollements, that
is on an induction on the length of a stratification, need to be
based on the case of derived simple rings. Hence it is of interest
to identify such rings or classes thereof.

\medskip
\begin{lem}\label{tracecriterion}
A semiperfect ring $R$ is derived simple provided that for each
finitely generated projective $R$-module $P$ the trace of $P$ in
$R$ equals $R$.
\end{lem}
\begin{proof}
The condition means that $\add R=\add P$ for any finitely generated
projective $R$-module $P$. So, for any two non-zero finitely
generated projective $R$-modules $P,Q$ there is a  non-zero
homomorphism $P\to Q$ which  maps an indecomposable summand of $P$
isomorphically onto an indecomposable summand of $Q$,  and is zero
elsewhere. Hence a compact complex (of finitely generated
projective $R$-modules) must have self-extensions unless it is
just a projective module, up to shift. Then all compact
exceptional complexes generate $\Dcal(R)$ and are  tilting
complexes. By Theorem \ref{single} it follows that $R$ is derived
simple.
\end{proof}

\medskip
\begin{prop}\label{herders}
Let $R$ be a semihereditary ring. Then the following assertions
are equivalent.
\begin{enumerate}
\item
 $R$ is derived simple.

\item
 Every non-zero finitely presented exceptional  module is tilting.

\item
 The universal
localization $\lambda: R \rightarrow R_T$ at any non-zero finitely
presented exceptional $R$-module $T$ vanishes.

\end{enumerate}
If $R$ satisfies these conditions, then for each finitely
generated projective $R$-module $P$ the trace of $P$ in $R$ equals
$R$.
\end{prop}
\begin{proof}
(1)$\Rightarrow$(2): By Theorem \ref{HRSextended}, every non-zero
finitely presented exceptional  module $X$ gives rise to a
recollement  \[
\xy (-45,0)*{\Dcal(B)}; {\ar (-25,2)*{}; (-35,2)*{}};
{\ar (-35,0)*{}; (-25,0)*{}^{}}; {\ar(-25,-2)*{}; (-35,-2)*{}};
(-15,0)*{\Dcal(A)}; {\ar (5,2)*{}; (-5,2)*{}_{}};
{\ar (-5,0)*{}; (5,0)*{}}; {\ar (5,-2)*{}; (-5,-2)*{}};
(15,0)*{\Dcal(C)};
\endxy
\]

\vspace{0.2cm}
where  $C=\End_R(X)$. So $\Dcal(C)\cong \Tria X\not=0$, hence
$\Dcal(B)=0$. But this means that $X$ generates $\Dcal(R)$ and is
thus a tilting module.

(2)$\Rightarrow$(3) follows immediately from Theorem
\ref{HRSextended}.

(3) $\Rightarrow$(1): Given a recollement \[
\xy (-45,0)*{\Dcal(B)}; {\ar (-25,2)*{}; (-35,2)*{}};
{\ar (-35,0)*{}; (-25,0)*{}^{}}; {\ar(-25,-2)*{}; (-35,-2)*{}};
(-15,0)*{\Dcal(A)}; {\ar (5,2)*{}; (-5,2)*{}_{j_!}};
{\ar (-5,0)*{}; (5,0)*{}}; {\ar (5,-2)*{}; (-5,-2)*{}};
(15,0)*{\Dcal(C)};
\endxy
\]

\vspace{0.2cm}
with $\Dcal(C)\not=0$, we know from Theorem \ref{single} that $X=
j_!(C)$ is a non-zero compact exceptional object, hence a direct
sum of shifts of
 finitely presented,
exceptional modules. Let $T$ be  one of these modules. By Theorem
\ref{HRSextended}, there is a recollement of $\Dcal(R)$ by $\Tria
T$ and $\Dcal(R_T)$ which must be trivial by  condition (3). Since
$\Tria X$ contains  $\Tria T$,  it follows that the recollement
above must  also be trivial.

\smallskip

Finally, the additional statement is condition (3) in the special
case when $T$ is projective, cf. the first case in Theorem
\ref{HRSextended}.
\end{proof}

\medskip
Now we obtain several examples.

\begin{prop}\label{exdersimple}
(1) All local rings
are derived simple.

(2) A right artinian hereditary ring is derived simple if and only
if it is simple artinian.

(3) A commutative semihereditary ring is derived simple if and
only if it is a Pr\"ufer domain.  In particular, $\Z$ and
polynomial rings in one variable  over fields are derived simple.

(4) A von Neumann regular ring R is derived simple if and only if
for each finitely generated projective module $P$, the trace of $P$ in $R$ equals $R$.
\end{prop}

\begin{proof}
 (1) follows immediately from Lemma \ref{tracecriterion}.

 (2) Simple artinian rings satisfy the criterion in  Lemma \ref{tracecriterion} and are therefore derived simple. Conversely,
 if a right artinian hereditary ring is derived simple, then we know from Proposition \ref{herders} that all indecomposable finitely generated projective modules are tilting, which shows that there is just one projective module up to isomorphism, which must of course be simple. Then $R$ is simple artinian.

 (3)
As shown in \cite[2.44]{Lam}, if  $P$ is a finitely generated
projective module over a commutative ring $R$, then
$R=\tau_P(R)\oplus {\rm ann}_R(P)$, where ${\rm ann}_R(P)=\{r\in
R\,\mid\,xr=0 \,\text{for all}\,x\in P\}$. So,  every commutative
semihereditary ring $R$ which is derived simple must be a domain.
In fact, if $x\in R\setminus \{0\}$, then by assumption $P=xR$ is
a finitely generated projective module, so we infer from
Proposition \ref{herders} that ${\rm ann}_R(x)={\rm ann}_R(P)=0$.
Conversely, every Pr\"ufer domain  $R$ satisfies condition (3) in
Proposition \ref{herders}. Indeed, if $T$ is  a non-zero finitely
presented exceptional $R$-module, then $T$ is projective by
\cite[2.2]{CM}. Furthermore, since $T$ is faithful, $R=\tau_T(R)$,
so the universal localization at $T$ is trivial.

(4) Recall that $R$ is a semihereditary ring of  weak global dimension zero \cite[2.32 and 4.21]{Lam}. So, the only-if-part follows from Proposition \ref{herders}. Moreover, all finitely
presented modules are finitely generated projective. But then
every non-zero compact exceptional object in $\Dcal(R)$ is a
direct sum of shifts of finitely generated, projective modules.
Thus we can prove the if-part arguing as in Proposition
\ref{herders}, (3) $\Rightarrow$ (1).
\end{proof}

Wiedemann \cite{W} has shown that derived simplicity with respect
to $\Dcal^b$ is a non-trivial property  for finite dimensional
algebras, going much beyond local algebras; he found an algebra
with two simple modules that is derived simple. Happel \cite{Hap1}
even showed derived simplicity, also with respect to $\Dcal^b$, for a
series of algebras with two simple modules and of finite global
dimension.

\bigskip
\section{A Jordan H\"older theorem for derived categories of
hereditary artin algebras}
\label{sectionproof}

In this section we state and prove the main result of this article, which covers hereditary artin algebras and  algebras which are
derived equivalent to them.

\medskip

Here by a {\em stratification} of the derived category $\Dcal(A)$
of a ring $A$ we mean a sequence of iterated recollements of the
following form: a recollement of $A$, if it is not derived simple,
\[
\xy (-45,0)*{\Dcal(B)}; {\ar (-25,2)*{}; (-35,2)*{}};
{\ar (-35,0)*{}; (-25,0)*{}^{}}; {\ar(-25,-2)*{}; (-35,-2)*{}};
(-15,0)*{\Dcal(A)}; {\ar (5,2)*{}; (-5,2)*{}_{}};
{\ar (-5,0)*{}; (5,0)*{}}; {\ar (5,-2)*{}; (-5,-2)*{}};
(15,0)*{\Dcal(C)};
\endxy
\] and
a recollement of $B$, if it is not derived simple,
\[
\xy (-45,0)*{\Dcal(B_1)}; {\ar (-25,2)*{}; (-35,2)*{}};
{\ar (-35,0)*{}; (-25,0)*{}^{}}; {\ar(-25,-2)*{}; (-35,-2)*{}};
(-15,0)*{\Dcal(B)}; {\ar (5,2)*{}; (-5,2)*{}_{}};
{\ar (-5,0)*{}; (5,0)*{}}; {\ar (5,-2)*{}; (-5,-2)*{}};
(15,0)*{\Dcal(B_2)};
\endxy
\]
and a recollement of $C$, if it is not derived simple,
\[
\xy (-45,0)*{\Dcal(C_1)}; {\ar (-25,2)*{}; (-35,2)*{}};
{\ar (-35,0)*{}; (-25,0)*{}^{}}; {\ar(-25,-2)*{}; (-35,-2)*{}};
(-15,0)*{\Dcal(C)}; {\ar (5,2)*{}; (-5,2)*{}_{}};
{\ar (-5,0)*{}; (5,0)*{}}; {\ar (5,-2)*{}; (-5,-2)*{}};
(15,0)*{\Dcal(C_2)};
\endxy
\]
and recollements of $B_i$ and of $C_i$ ($i=1,2$), if they are not
derived simple, and so on, until we arrive at derived simple rings
at all positions, or continue ad infinitum.

\medskip
\begin{thm} \label{maintheorem}
Let $A$ be derived equivalent to a hereditary artin algebra and let
$S_1, \dots, S_n$ be representatives of the isomorphism classes of
simple $A$-modules. Denote the endomorphism rings by
$D_i:=\End_A(S_i)$. Then $\Dcal(A)$ has a stratification whose $n$
factors are the categories $\Dcal(D_i)$. Any stratification of
$\Dcal(A)$ has precisely these factors, up to ordering and derived
equivalence.
\end{thm}


Note that derived equivalence for the skew-fields $D_i$ just means Morita equivalence.

\begin{proof} Without loss of generality, we assume $A$ is hereditary.
We will proceed by induction on the number $n$ of isomorphism classes of
non-isomorphic simple modules of the algebra $A$.

\medskip

Any hereditary
artin algebra has a standard stratification of length $n$ whose
factors are precisely the endomorphism ring of the simple modules.
Indeed, a simple projective module $S_1$ generates an ideal $J_1$ that is
projective on both sides - more precisely it is a heredity ideal and thus
a stratifying ideal - and hence there is a recollement involving $A$,
$D_1 = \End_A(S_1)$ and $A/J_1$, which again is a hereditary artin algebra,
with simples $S_2, \dots, S_n$.

\medskip

For uniqueness we will prove a stronger result by induction using Proposition \ref{completingrecoll}: 
If $A$ is a hereditary artin algebra, any stratification of $\Dcal(A)$ can be rearranged into a finite chain of increasing derived module categories of 
hereditary artin algebras
\vspace{0.1cm}
\[
\xy (-58,0)*{\Dcal(A_n)}; {\ar (-44,2)*{}; (-50,2)*{}}; {\ar
(-50,0)*{}; (-44,0)*{}}; {\ar(-44,-2)*{}; (-50,-2)*{}};
(-34,0)*{\Dcal(A_{n-1})}; {\ar (-18,2)*{}; (-24,2)*{}}; {\ar
(-24,0)*{}; (-18,0)*{}}; {\ar(-18,-2)*{}; (-24,-2)*{}};
(-8,0)*{\ldots}; {\ar (8,2)*{}; (2,2)*{}}; {\ar (2,0)*{};
(8,0)*{}}; {\ar (8,-2)*{}; (2,-2)*{}}; (15,0)*{\Dcal(A_2)}; {\ar
(28,2)*{}; (22,2)*{}}; {\ar (22,0)*{}; (28,0)*{}}; {\ar
(28,-2)*{}; (22,-2)*{}}; (36,0)*{\Dcal(A_1)};
\endxy \vspace{0.1cm}
\]
of length $n$, where $A_1=A$. Moreover this chain is induced by a sequence of homological epimorphisms $A_1\ra A_2 \ra \ldots \ra A_{n-1} \ra A_n$,
such that $A_n$ is derived simple and for each $1\leq i \leq n-1$, 
\vspace{0.1cm}
\[
\xy (-45,0)*{\Dcal(A_{i+1})}; {\ar (-25,2)*{}; (-35,2)*{}};
{\ar (-35,0)*{}; (-25,0)*{}^{}}; {\ar(-25,-2)*{}; (-35,-2)*{}};
(-16,0)*{\Dcal(A_{i})};
\endxy \vspace{0.1cm}
\]
can be completed to a full recollement with the third term being
the derived module category of some derived simple algebra, say, $E_i$. We write $E_n=A_n$ for convenience. 
These $E_i$'s are the endomorphism rings of simple $A$-modules and $\Dcal(E_i)$ ($i=1,\ldots,n$) are precisely the 
derived simple factors in the original stratification.

\medskip

When $n=1$, the hereditary algebra $A$ has only one
simple module. This simple module is also projective, so $A$ is
Morita equivalent to a skew-field, and hence derived simple, cf.~Proposition \ref{exdersimple}.
In the following we assume $n \geq 2$. Suppose a stratification of $\Dcal(A)$ starts with
\[
(1) \qquad \xy (-45,0)*{\Dcal(B)}; {\ar (-25,2)*{}; (-35,2)*{}};
{\ar (-35,0)*{}; (-25,0)*{}^{}}; {\ar(-25,-2)*{}; (-35,-2)*{}};
(-15,0)*{\Dcal(A)}; {\ar (5,2)*{}; (-5,2)*{}_{}};
{\ar (-5,0)*{}; (5,0)*{}}; {\ar (5,-2)*{}; (-5,-2)*{}};
(15,0)*{\Dcal(C)};
\endxy
\]
and, if $C$ is not derived simple,
\[
(2) \qquad  \xy (-45,0)*{\Dcal(C_1)}; {\ar (-25,2)*{}; (-35,2)*{}};
{\ar (-35,0)*{}; (-25,0)*{}^{}}; {\ar(-25,-2)*{}; (-35,-2)*{}};
(-15,0)*{\Dcal(C)}; {\ar (5,2)*{}; (-5,2)*{}_{}};
{\ar (-5,0)*{}; (5,0)*{}}; {\ar (5,-2)*{}; (-5,-2)*{}};
(15,0)*{\Dcal(C_2).};
\endxy
\]
Applying Proposition \ref{completingrecoll}, we can rearrange the
two recollements into
\[
(3)\qquad  \xy (-45,0)*{\Dcal(B')}; {\ar (-25,2)*{}; (-35,2)*{}};
{\ar (-35,0)*{}; (-25,0)*{}^{}}; {\ar(-25,-2)*{}; (-35,-2)*{}};
(-15,0)*{\Dcal(A)}; {\ar (5,2)*{}; (-5,2)*{}_{}};
{\ar (-5,0)*{}; (5,0)*{}}; {\ar (5,-2)*{}; (-5,-2)*{}};
(15,0)*{\Dcal(C_2)};
\endxy
\]
and
\[
(4) \qquad \xy (-45,0)*{\Dcal(B)}; {\ar (-25,2)*{}; (-35,2)*{}};
{\ar (-35,0)*{}; (-25,0)*{}^{}}; {\ar(-25,-2)*{}; (-35,-2)*{}};
(-15,0)*{\Dcal(B')}; {\ar (5,2)*{}; (-5,2)*{}_{}};
{\ar (-5,0)*{}; (5,0)*{}}; {\ar (5,-2)*{}; (-5,-2)*{}};
(15,0)*{\Dcal(C_1)};
\endxy
\]

\vspace{0.2cm}
where $C_1=E$, $C_2=F$ and $B'=G$ in the notation of Proposition \ref{completingrecoll}. Note that under the rearrangement the factors
$\Dcal(B)$, $\Dcal(C_1)$ and $\Dcal(C_2)$ are preserved.
By Theorem \ref{single}, the image $j_!(C_2)$ of $C_2$ under the full embedding on the upper right 
corner of the recollement $(3)$ is compact and exceptional in $\Dcal(A)$. By Corollary \ref{HRSextended-more} we
can then assume that $B'$ is a hereditary artin algebra and the
recollement $(3)$ is induced by a homological epimorphism $A\ra B'$. Using the same argument on the  recollement $(4)$ 
we can assume that $B$ is a hereditary artin algebra and the recollement $(4)$ is induced by a homological epimorphism $B'\ra B$. 

\medskip
The original stratification of $\Dcal(A)$ is given by the recollement $(1)$ and a stratification on $\Dcal(B)$ and 
on $\Dcal(C)$ respectively.
By iterating the above procedure we can transport the recollements in the stratification on $\Dcal(C)$ to the left hand side of 
$\Dcal(A)$ and thus obtain a chain of increasing derived module categories of hereditary artin algebras
\vspace{0.1cm}
\[
\xy (-60,0)*{\Dcal(B)}; {\ar (-44,2)*{}; (-50,2)*{}}; {\ar
(-50,0)*{}; (-44,0)*{}}; {\ar(-44,-2)*{}; (-50,-2)*{}};
(-34,0)*{\Dcal(B_1')}; {\ar (-18,2)*{}; (-24,2)*{}}; {\ar
(-24,0)*{}; (-18,0)*{}}; {\ar(-18,-2)*{}; (-24,-2)*{}};
(-8,0)*{\Dcal(B_2')}; {\ar (8,2)*{}; (2,2)*{}}; {\ar (2,0)*{};
(8,0)*{}}; {\ar (8,-2)*{}; (2,-2)*{}}; (15,0)*{\ldots}; {\ar
(28,2)*{}; (22,2)*{}}; {\ar (22,0)*{}; (28,0)*{}}; {\ar
(28,-2)*{}; (22,-2)*{}}; (38,0)*{\Dcal(A)};
\endxy \vspace{0.1cm}
\]
induced by a sequence of homological epimorphisms $A \ra \ldots
\ra B_2' \ra B_1' \ra  B_0'=B$. The subfactors in the chain are precisely the derived simple 
factors in the stratification of $\Dcal(C)$. Moreover, each $B_i'$ is a partial tilting module over $B_{i+1}'$
($i \ge 0$). Hence the number of non-isomorphic simple modules of
$B_i'$ is strictly smaller than that of $B_{i+1}'$. This implies
that the above chain, and hence the stratification of $\Dcal(C)$, must have finite length. 

\medskip

Since $B$ is a hereditary artin algebra and has a smaller number of non-isomorphic simple modules than $A$, we can apply induction 
and assume the stratification on $\Dcal(B)$ has been rearranged as desired. 
Combining this with the chain obtained in the previous paragraph, we can rearrange
the original stratification on $\Dcal(A)$ into a finite chain of increasing derived module categories of hereditary artin algebras
\[
\xy (-40,0)*{\Dcal(A_m)}; {\ar (-24,2)*{}; (-30,2)*{}}; {\ar
(-30,0)*{}; (-24,0)*{}}; {\ar(-24,-2)*{}; (-30,-2)*{}};
(-17,0)*{\ldots}; {\ar (-4,2)*{}; (-10,2)*{}}; {\ar (-10,0)*{};
(-4,0)*{}}; {\ar (-4,-2)*{}; (-10,-2)*{}}; (6,0)*{\Dcal(A_2)}; {\ar
(22,2)*{}; (16,2)*{}}; {\ar (16,0)*{}; (22,0)*{}}; {\ar
(22,-2)*{}; (16,-2)*{}}; (33,0)*{\Dcal(A_1)};
\endxy \vspace{0.1cm}
\]
of length, say, $m$, induced by a sequence of homological epimorphisms $A=A_1\ra A_2\ra \ldots \ra A_m$, and such that the factors $\Dcal(E_i)$ ($i=1,\ldots,m$) are precisely the derived simple factors 
in the original stratification on $\Dcal(A)$. 

\medskip
Consider the first recollement
\[
\xy (-45,0)*{\Dcal(A_2)}; {\ar (-25,2)*{}; (-35,2)*{}};
{\ar (-35,0)*{}; (-25,0)*{}^{}}; {\ar(-25,-2)*{}; (-35,-2)*{}};
(-15,0)*{\Dcal(A_1)}; {\ar (5,2)*{}; (-5,2)*{}_{j_!}};
{\ar (-5,0)*{}; (5,0)*{}}; {\ar (5,-2)*{}; (-5,-2)*{}};
(15,0)*{\Dcal(E_1)};
\endxy \vspace{0.1cm}
\]
taken from the right hand side of the above filtration. By Theorem
\ref{single}, $X=j_!(E_1)$ is a compact exceptional object in
$\Dcal(A)$.
 We claim that $X$ is indecomposable. Indeed, as
explained in Subsection \ref{exceptseq}, the indecomposable summands of $X$ can be arranged into an exceptional sequence.
Therefore, $X$ has a triangular (directed) endomorphism ring $\End_A(X)\simeq
E_1$. So $E_1$ has a simple projective module, generating a stratifying
ideal $J$ and thus inducing a recollement for $E_1$. But $E_1$ is derived
simple and the recollement must be trivial. Thus $J=E_1$ and $E_1$ is a
simple algebra, which implies the claim.

\medskip

Now since $X$ is an indecomposable, finitely
presented and exceptional module, by Theorem
\ref{HRSextended} $A_2$ can be chosen to be the
hereditary artin algebra obtained from universal localization of $A_1$ at
$X$. Note that $A_2$ has $n-1$ simple modules by Proposition
\ref{artin}. By induction hypothesis, we see that $m=n$
and that the derived simple algebras $E_2,\ldots,E_n$ in the
stratification are the endomorphism rings of the simple
$A_2$-modules.

\medskip
If $X$ is projective, $X/\text{Rad}(X)$ is a simple $A$-module. Therefore up to renumbering we have $\End_A(X/\text{Rad}(X)) \simeq D_1$.
Since $\End_A(X/\text{Rad}(X)) \simeq \End_A(X)$, we have $E_1\simeq D_1$. The simple $A_2$-modules are
precisely those simple $A$-modules that are not isomorphic to
$X/\text{Rad}(X)$,  so $\{{E_1\simeq \End_A(X)},\, E_2,\ldots,\,
E_n\}=\{D_1,\ldots,D_n\}$. If $X$ has projective dimension  one,
then  $(A_2)_A$  complements $X$ to a tilting module $T=A_2\oplus
X$, so by Theorem \ref{endotilting} the endomorphism rings of the
indecomposable summands of $T$ are precisely $D_1,\ldots,D_n$. As
the  endomorphism rings  of the indecomposable summands of
$(A_2)_A$ coincide with the endomorphism rings of the simple
$A_2$-modules, we conclude also in this case that $\{{E_1\simeq
\End_A(X)},\, E_2,\ldots,\, E_n\}=\{D_1,\ldots,D_n\}$. \end{proof}


\medskip

We are now ready to answer the question that has been stated after
Theorem \ref{endotilting},
about the endomorphism rings of indecomposable direct summands in a
tilting complex. Recall that an object $T$ in $\Dcal(A)$ is called a
{\it tilting complex} if $T$ is compact, exceptional, and $\Dcal(A)$
equals $\Tria T$, the smallest triangulated category containing $T$
and closed under small coproducts.

\begin{cor} Let $A$ be a hereditary artin algebra,
and $T$ a multiplicity free tilting
complex in $\Dcal(A)$. Then the endomorphism rings of the
indecomposable direct summands of $T$ are precisely those of the
non-isomorphic simple modules.
\label{endoftiltingcomplex}
\end{cor}

\begin{proof} By (\ref{exceptseq}), the indecomposable direct
summands of $T$ form a complete exceptional sequence, say $(T_1,
T_2, \ldots, T_n)$. From the proof of Corollary
\ref{HRSextended-more}, this exceptional sequence induces a
stratification of $\Dcal(A)$ whose factors are the derived module
categories of $\End_A(T_i)$ ($i=1,2,\ldots, n$). Due to Theorem
\ref{maintheorem}, these endomorphism rings are up to derived
equivalence the endomorphism rings of the non-isomorphic simple
$A$-modules. But for skew-fields, derived equivalence implies Morita
equivalence. The $T_i$ being indecomposable then implies their
endomorphism rings are local and hence isomorphic to those of the
simple modules $S_i$.
\end{proof}

\bigskip
\section{What can fail}
\label{sectionex}

In this Section we first explain why only derived categories of rings
should be permitted as outer terms of recollements in our context. We also
give an example showing that Theorem \ref{maintheorem} fails without
finiteness assumptions - while we do not have examples of failure
for artin algebras in general, that is, when dropping the assumption
'hereditary'.
\bigskip

Which kind of recollements is meaningful when trying to prove a Jordan H\"older
theorem for derived categories of rings? Since recollement is a natural
concept for triangulated categories in general, a natural first choice is
is to admit all triangulated categories as terms in a
recollement. This choice, however, leads to an abundance of recollements,
for instance in the following way: By the second Theorem in \cite[1.6]{AKL},
there exists a recollement of the derived category $\Dcal(R)$ as soon as there
exists an object $T_1$ generating a smashing subcategory. Thus, we may for
instance choose $T_1$ to be a finitely generated $R$-module. Then we will
get a recollement, where on the right hand side we get the triangulated
category generated by $T_1$. Under some assumptions (see \cite{AKL})
this category is equivalent to the derived category of the differential graded
endomorphism algebra of $T_1$ - which is an ordinary algebra only if $T_1$
has no self-extensions. We always get the derived category of
another differential graded algebra on the left hand side. Hence, making
such a generous choice for factors of recollements will imply that there are
few derived simple rings, and it will move the question of derived simplicity
to different kinds of triangulated categories. When considering recollements
on this general level, the terms in a 'composition series' of $\Dcal(R)$
usually will be triangulated categories that are much less accessible than
derived categories of rings, at least by current technology.

\medskip

Moreover, allowing general triangulated categories as factors of recollements
definitely produces counterexamples to a general form of Theorem
\ref{maintheorem}, even for very small and natural examples, as the
following example shows:

\begin{ex}
This example is taken from \cite[Example 5.1]{AKL}, where more
detail is given.

Let $A$ be the Kronecker algebra over an algebraically closed field $k$.
This is a hereditary
algebra with two simple modules, whose derived category is equivalent to
the category of coherent sheaves on a projective line, by \cite{Beilinson}.
It has obvious recollements, where the two factors each are equivalent to the
derived category of $\mathrm{Mod}\text{-}k$, which is clearly derived simple.

However, there is a rather different recollement of the following
form:

\[
\xy (-45,0)*{\Dcal(A_\tube)}; {\ar (-25,2)*{}; (-35,2)*{}};
{\ar (-35,0)*{}; (-25,0)*{}^{}}; {\ar(-25,-2)*{}; (-35,-2)*{}};
(-15,0)*{\Dcal(A)}; {\ar (5,2)*{}; (-5,2)*{}_{}};
{\ar (-5,0)*{}; (5,0)*{}}; {\ar (5,-2)*{}; (-5,-2)*{}};
(15,0)*{\Tria \tube};
\endxy \vspace{0.3cm}
\]
Here,  $A_\tube$ is a simple artinian ring, thus derived simple, but not Morita equivalent to
$k$. And the triangulated category $\Tria \tube$ on the right hand side,
generated by the regular modules, that is by the tubes in the Auslander
Reiten quiver, can
be decomposed further, since there are no maps or extensions between
different tubes - thus we can iterate forming recollements infinitely
many times, producing an infinite derived composition series.

\medskip

{\rm The terms of this recollement are obtained as follows:
We consider the class of indecomposable regular right
$A$-modules $\tube$. By the Auslander-Reiten formula the tilting class
$$\tube^\perp={}^o\tube$$ is the torsion class of all
{\em divisible   modules}. There exists a tilting module $W$
which generates $\tube^\perp$. The module $W$
can be chosen as the direct sum of a set of representatives of the Pr\"ufer
$A$-modules and the generic  $A$-module $G$. Moreover, there is an exact
sequence $$0\to A \to W_0\to W_1\to 0$$ where $W_0\cong G^d$, and $W_1$ is a
direct sum of  Pr\"ufer modules.
Then $W$ is equivalent to the tilting module $A_\tube\oplus A_\tube/A$, and
there is the above recollement,
where $A_\tube\cong \End _A(W_0)\cong {(\End _A(G))}^{d\times d}$.}
\end{ex}

Thus, a general version of \ref{maintheorem} would fail rather dramatically
even in this easy situation.

\bigskip

>From this discussion we can conclude that the question of validity
of a Jordan H\"older theorem has to be restricted to
stratifications with all factors being derived categories of
rings. We are left with the following problem, which like in the
classical situations has a negative answer - of course, some
finiteness assumptions are needed in \ref{maintheorem}.

\medskip
{\it Problem. Given a ring $A$, do all stratifications of $\Dcal(A)$
by derived module categories of rings have the same finite
number of factors, and are these factors the same for all stratifications,
up to ordering and up to derived equivalence?}
\medskip

As to be expected, on this level of generality, the problem has a negative
answer. The next example is a counterexample; it shows that the number of
factors may be infinite.

\medskip
\begin{ex}
Let $k$ be a field, and $A=k^{\mathbb N}$ the direct product of
countably many copies of $k$. Then $\Dcal(A)$ has an
infinitely long stratification. More precisely, it has a recollement
with itself occuring as one factor:
\vspace{0.1cm}
\[
\xy (-45,0)*{\Dcal(A)}; {\ar (-25,2)*{}; (-35,2)*{}}; {\ar
(-35,0)*{}; (-25,0)*{}^{}}; {\ar(-25,-2)*{}; (-35,-2)*{}};
(-15,0)*{\Dcal(A)}; {\ar (5,2)*{}; (-5,2)*{}_{}}; {\ar (-5,0)*{};
(5,0)*{}}; {\ar (5,-2)*{}; (-5,-2)*{}}; (15,0)*{\Dcal(k)};
\endxy
\]
\end{ex}

\smallskip
Let $e_1=(1,0,0,\dots) \in A$ be the idempotent supported on the
first index. Then $e_1A$ is finitely generated projective with  endomorphism ring $k$, and
 the universal localization of $A$ at $e_1A$ is a homological epimorphism since $A$ is von Neumann regular.  By 
 \cite[4.5]{AKL}, $e_1A$ induces a recollement. The ring on the left hand side is $A/\tau_{e_1A}(A)$, which is isomorphic to $A$ itself.



\medskip

More dramatically, uniqueness of factors in a finite stratification can fail and even the length of finite stratifications is not an invariant. Examples have been constructed by Chen and Xi \cite{CX}.

\bigskip


\end{document}